\documentclass[12pt]{article}
\usepackage{a4wide}
\usepackage{amsthm}
\usepackage{amsfonts}
\usepackage{amssymb}
\usepackage{amsmath}
\usepackage{cite}
\usepackage{epsfig}
\newtheorem{theorem}{Theorem}
\newtheorem{lemma}[theorem]{Lemma}
\newtheorem{proposition}[theorem]{Proposition}
\newtheorem{conjecture}{Conjecture}
\newcommand\NN{{\mathbb N}}
\newcommand\ZZ{{\mathbb Z}}
\newcommand\RR{{\mathbb R}}
\usepackage{mathtools}

\begin{document}
\title{Four-coloring Eulerian triangulations of the torus}

\author{Marcin Bria\'nski\thanks{Theoretical Computer Science Department, Faculty of Mathematics and Computer Science, Jagiellonian University, Krak\'ow, Poland. E-mail: \texttt{marcin.brianski@doctoral.uj.edu.pl}.}\and
        Daniel Kr{\'a}l'\thanks{Faculty of Informatics, Masaryk University, Botanick\'a 68A, 602 00 Brno, Czech Republic. E-mail: {\tt \{dkral,540987\}@fi.muni.cz}. Supported by the MUNI Award in Science and Humanities (MUNI/I/1677/2018) of the Grant Agency of Masaryk University.}\and
        \newcounter{lth}
	\setcounter{lth}{1}
	Ander Lamaison\thanks{Institute for Basic Science, 55 Expo-ro, Yuseong-gu, 34126 Daejeon, South Korea. E-mail: {\tt ander@ibs.re.kr}. Previous affiliation: Faculty of Informatics, Masaryk University, Botanick\'a 68A, 602 00 Brno, Czech Republic. This author was also supported by the MUNI Award in Science and Humanities (MUNI/I/1677/2018).}\and
	Xichao Shu$^\fnsymbol{lth}$}

\date{}

\maketitle

\begin{abstract}
Hutchinson, Richter and Seymour [J. Combin. Theory Ser. B 84 (2002), 225--239] showed that
every Eulerian triangulation of an orientable surface that has a sufficiently high representativity is $4$-colorable.
We give an explicit bound on the representativity in the case of the torus by proving that
every Eulerian triangulation of the torus with representativity at least $10$ is $4$-colorable.
We also observe that the bound on the representativity cannot be decreased to less than $8$ as
there exists a non-$4$-colorable Eulerian triangulation of the torus with representativity $7$.
\end{abstract}

\section{Introduction}
\label{sec:intro}

Graph coloring is one of the most classical and oldest concepts in graph theory.
Recall that the \emph{chromatic number} of a graph $G$, denoted by $\chi(G)$, is the smallest $k$ such that
$G$ is \emph{$k$-colorable},
i.e., it is possible to assign $k$ colors to the vertices of $G$ so that any pair of adjacent vertices receive distinct colors;
graphs with chromatic number equal to $k$ are referred to as \emph{$k$-chromatic}.
The Four Color Theorem~\cite{AppH76,RobSST97} asserts that every planar graph is $4$-colorable.
It is well-known that every Eulerian triangulation of the plane, i.e., a triangulation with all degrees even, is $3$-colorable~\cite{Hea98}, and
in fact, a planar graph is $3$-colorable if and only if it is a subgraph of an Eulerian triangulation of the plane.

We are interested in graphs that can be embedded in surfaces of higher genera in a locally planar way defined further;
we refer the reader to the monograph~\cite{MohT01} for the comprehensive introduction to graphs embeddable in surfaces.
The \emph{representativity} of an embedded graph $G$, denoted by $r(G)$,
is the smallest number of intersections of a non-contractible curve with the embedding of $G$.
Thomassen~\cite{Tho93} proved that for every surface $\Sigma$,
there exists $r_\Sigma$ such that every graph embedded in $\Sigma$ with representativity at least $r_\Sigma$ is $5$-colorable;
this also follows from a more general (later) result also by Thomassen~\cite{Tho97} that
the number of $6$-critical graphs embeddable in any fixed surface is finite (recall that a graph $G$ is \emph{$k$-critical}
if $G$ is $k$-chromatic but every proper subgraph of $G$ is $(k-1)$-colorable).

A similar phenomenon as in the case of plane triangulations where Eulerian triangulation are $3$-colorable,
i.e., one less color is needed,
was conjectured by Collins and Hutchinson~\cite{ColH99}
to hold for locally planar triangulations of orientable surfaces of higher genera.
The conjecture of Collins and Hutchinson
was proven by Hutchinson, Richter and Seymour~\cite{HutRS02}
who showed that Eulerian triangulations of orientable surfaces that have large representativity are $4$-colorable.

\begin{theorem}[{Hutchinson, Richter and Seymour~\cite{HutRS02}}]
\label{thm:gen}
For every orientable surface $\Sigma$, there exists $r_\Sigma$ such that
every Eulerian triangulation of $\Sigma$ with representativity at least $r_\Sigma$ is $4$-colorable.
\end{theorem}

Since it is not hard to construct $4$-chromatic Eulerian triangulations of the torus with large representativity,
the bound on the chromatic number given in Theorem~\ref{thm:gen} is the best possible.
Note that there are $5$-chromatic Eulerian triangulations of the projective plane with large representativity~\cite{HutRS02},
which follows from the characterization of the chromatic number of Eulerian triangulations of the projective plane by Mohar~\cite{Moh02}; in particular,
any Eulerian triangulation of the projective plane obtained by adding a vertex to each face of a $4$-chromatic quadrangulation of the projective plane is $5$-chromatic.
Hence, the condition that the surface $\Sigma$ is orientable cannot be omitted in Theorem~\ref{thm:gen}.
We remark that Nakamoto~\cite{Nak08} showed that
every $5$-chromatic Eulerian triangulation of a non-orientable surface with large representativity
contains an independent set whose removal results in an even-faced non-$3$-colorable graph.
We also remark that there are Eulerian triangulations of the Klein bottle that are $6$-chromatic,
however, every $6$-chromatic Eulerian triangulation of the Klein bottle contains $K_6$ as a subgraph~\cite{KraMNPS12}.

In their paper~\cite{HutRS02}, Hutchinson, Richter and Seymour did not provide any explicit bound on $r_{\Sigma}$ and
the bounds that follow from their proof would be enormous for surfaces with large genera as the arguments
are based on results from the graph minor project.
Our main result is a reasonably small bound on $r_\Sigma$ in Theorem~\ref{thm:gen} in the case when $\Sigma$ is the torus;
as discussed further, the bound of $10$ in the statement of the next theorem cannot be improved to less than $8$.

\begin{theorem}
\label{thm:main}
Every Eulerian triangulation of the torus with representativity at least $10$ is $4$-colorable.
\end{theorem}

\noindent We remark that in the case of the torus,
combining proof techniques used in~\cite{HutRS02} with results of Schrijver from~\cite{Sch92,Sch93}
would likely give a two-digit bound in the case of the torus
although we have not computed the bound exactly (the exact value would also depend on the amount of fine tuning of the arguments).

\begin{figure}
\begin{center}
\epsfbox{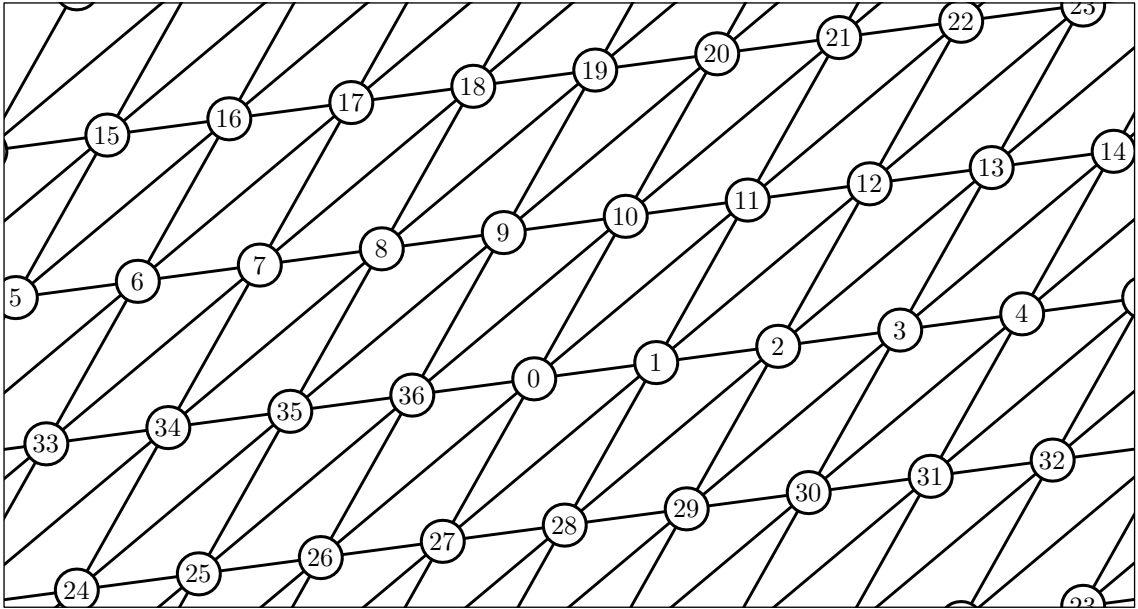}
\end{center}
\caption{The triangulation $\Gamma(\ZZ_{37},\{1,10,11\})$ of the torus.}
\label{fig:Z37}
\end{figure}

We now briefly overview the above mentioned lower bound of $8$ on $r_\Sigma$ from Theorem~\ref{thm:gen} in the case when $\Sigma$ is the torus
while deferring more detailed discussion to Section~\ref{sec:lower}.
Thomassen~\cite{Tho94} classified $6$-critical toroidal graphs and
derived as a corollary~\cite[Theorem 3.3]{Tho94} that
a $6$-regular triangulation of the torus is $5$-colorable
unless the triangulation is $K_7$ or the Cayley graph $\Gamma(\ZZ_{11},4,5)$.
Using Altshuler's classification of $6$-regular toroidal graphs~\cite{Alt73},
Yeh and Zhu~\cite{YehZ03} classified all $6$-regular toroidal graphs with chromatic number larger than three.
In particular, the $6$-regular Cayley graph $\Gamma(\ZZ_{37},\{1,10,11\})$, which is depicted in Figure~\ref{fig:Z37}, is $5$-chromatic.
As we discuss in Section~\ref{sec:lower},
the representativity of the (unique) Eulerian triangulation of the torus given by this graph is $7$ and
so the bound in Theorem~\ref{thm:main} cannot be decreased to less than $8$.

We next present the overview of the proof of Theorem~\ref{thm:main} and discuss how the paper is structured.
Similarly to~\cite{HutRS02},
we analyze cycles with specific free homotopies and extensions of their precolorings with three colors.
Rather than introducing necessary concepts from~\cite{HutRS02},
we decided to cast the results needed for our proof in the language of vertex-colorings and
included short proofs for completeness in Section~\ref{sec:main};
we would like to remark that extendibility of $3$-colorings to partial Eulerian triangulations of the plane
has been studied using various notations~\cite{DikKK02,Kro72,Kro73}, also see the survey~\cite{Ste93}.
In Section~\ref{sec:main}, we also prove our main result---Theorem~\ref{thm:main}.
To prove Theorem~\ref{thm:main},
we need to understand conditions on the existence of vertex-disjoint cycles with specific free homotopies.
In Section~\ref{sec:polygon}, we obtain such understanding using classical results of Schrijver~\cite{Sch92,Sch93}
on the existence of vertex-disjoint cycles with prescribed free homotopies in graphs embedded in the torus.
As in each of Section~\ref{sec:polygon} and Section~\ref{sec:main} we need concepts and notation specific for that section only,
we introduce these briefly at the beginning of each of the two sections.
In Section~\ref{sec:lower},
we show that the representativity of the (unique) Eulerian triangulation of the torus given by the Cayley graph $\Gamma(\ZZ_{37},\{1,10,11\})$ is $7$ and
we give a short proof that the graph is $5$-chromatic for completeness.
We conclude in Section~\ref{sec:concl} by discussing the limits of the techniques applied in Section~\ref{sec:polygon}.
We also conjecture that the bound on the representativity given by the Cayley graph $\Gamma(\ZZ_{37},\{1,10,11\})$ is tight,
i.e., the bound in Theorem~\ref{thm:main} can be improved to $8$.

\section{Disjoint cycles on the torus}
\label{sec:polygon}

We start with recalling results from~\cite{Sch92,Sch93}
on the existence of disjoint homotopic cycles in graphs embedded in the torus;
we refer to~\cite{Sch92} for particular results presented below.
We view the torus as $\RR^2/\ZZ^2$ and
let $C_{m,n}$ for $(m,n)\in\ZZ^2$ be the closed curve in the torus parameterized as $(m\cdot t\mod 1,n\cdot t\mod 1)$ for $t\in[0,1]$.
Every closed curve in the torus is freely homotopic to $C_{m,n}$ for some $(m,n)\in\ZZ^2$, see e.g.~\cite{Sti12}.

Fix a graph $G$ embedded in the torus for our presentation until Theorem~\ref{thm:polar}, and
let $c_G(m,n)$ be the minimum number of intersections of a closed curve freely homotopic to $C_{m,n}$ with the graph $G$.
Note that $c_G(0,0)=0$ and
the representativity of $G$ is the minimum value of $c_G(m,n)$ taken over all non-zero pairs $(m,n)\in\ZZ^2$.
Define a convex set $P(G)$ as
\[P(G)=\{(x,y)\in\RR^2\mbox{ such that }mx+ny\le c_G(m,n)\}.\]
Note that
the set $P(G)$ is \emph{symmetric} and bounded;
the former follows from that $c_G(m,n)=c_G(-m,-n)$ for all $(m,n)\in\ZZ^2$, and
the latter from finiteness of $c_G(0,1)=c_G(0,-1)$ and $c_G(1,0)=c_G(-1,0)$.
It can be shown~\cite[Section 4]{Sch92} that
\begin{equation}
\max\{mx+ny\mbox{ such that }(x,y)\in P(G)\}=c_G(m,n)
\label{eq:maxP}
\end{equation}
for any $(m,n)\in\ZZ^2$.
The following theorem of Schrijver~\cite{Sch92}
relates the existence of disjoint cycles with a given free homotopy to the set $P(G)$.

\begin{theorem}[{Schrijver~\cite[Theorem 2]{Sch92}}]
\label{thm:Sch}
Let $G$ be a graph embedded in the torus.
For any non-zero $(m,n)\in\ZZ^2$ such that $m$ and $n$ are coprime and any $k\in\NN$,
it holds that $(kn,-km)\in P(G)$ if and only if
the graph $G$ contains $k$ pairwise disjoint cycles freely homotopic to $C_{m,n}$.
\end{theorem}

Recall that the polar of the set $P(G)$ is the set $P(G)^*$ defined as
\[P(G)^*=\{(x',y')\in\RR^2\mbox{ such that }xx'+yy'\le 1\mbox{ for all }(x,y)\in P(G)\}.\]
If $X$ is a compact symmetric subset of $\RR^2$, we define $\lambda(X)$ as
\[\lambda(X)=\min\{t\ge 0\mbox{ such that } t\cdot X\mbox{ contains a non-zero integer point}\}.\]
It can be shown~\cite[Section 5]{Sch92} that $\lambda(P(G)^*)$ is equal to the representativity $r(G)$ of the graph $G$.
Since it holds for any compact symmetric subset of $\RR^2$ that $\lambda(X)\lambda(X^*)\le 4/3$~\cite[Theorem 4]{Sch92},
we obtain that $\lambda(P(G))^{-1}\ge 3r(G)/4$, which is equivalent to the following theorem.

\begin{theorem}
\label{thm:polar}
Let $G$ be a graph embedded in the torus.
There exists a non-zero $(a,b)\in\ZZ^2$ such that $\left(\frac{3r(G)}{4}\cdot a,\frac{3r(G)}{4}\cdot b\right)\in P(G)$.
\end{theorem}

To apply Theorems~\ref{thm:Sch} and~\ref{thm:polar},
we will need the following auxiliary lemma on the containment of certain points
in convex symmetric sets in the plane that contain certain other points.

\begin{lemma}
\label{lm:polygon}
Let $P$ be a convex symmetric set in the plane.
If $P$ contains the point $(0,7.5)$,
a point $(x,-10)$ for some $x\in\RR$,
a point $(10,-y)$ for some $y\in [0,5]$ and
a point $(t,10-t)$ for some $t\in\RR$,
then $P$ contains one of the following quadruples of points:
\begin{itemize}
\item the points $(0,4)$, $(4,-4)$, $(4,0)$ and $(4,4)$, or
\item the points $(0,-4)$, $(4,-8)$, $(4,-4)$ and $(4,0)$.
\end{itemize}
\end{lemma}

\begin{proof}
Fix $y\in [0,5]$, $x\in\RR$ and $t\in\RR$.
Note that the set $P$ contains the following eight points:
$(0,7.5)$, $(0,-7.5)$, $(x,-10)$, $(-x,10)$, $(10,-y)$, $(-10,y)$, $(t,10-t)$ and $(-t,t-10)$.
We will be visualizing the points contained in the set $P$ on a square unit grid,
which will always contain the points $(-10,-10)$ and $(10,10)$;
we refer to Figure~\ref{fig:poly1} for an example.
The points $(0,4)$ and $(0,-4)$ are convex combinations of the points $(0,-7.5)$ and $(0,7.5)$ and
so they are contained in $P$.
The points $(4,-4)$ and $(4,0)$
are contained in the triangle with the vertices $(0,-7.5)$, $(0,7.5)$ and $(10,-y)$ for any $y\in [0,5]$ (see Figure~\ref{fig:poly1} for illustration) and
so they both are also contained in $P$.
Hence, we need to show that the set $P$ contains the point $(4,-8)$ or the point $(4,4)$.
To do so, we will be successively restricting the values of $x$, $y$ and $t$;
to present our arguments as simply as possible,
we will restrict ranges of not yet analyzed values of $x$, $y$ and $t$ to closed intervals and
visualize our arguments in Figures~\ref{fig:poly1}--\ref{fig:poly10}.

\begin{figure}
\begin{center}
\epsfbox{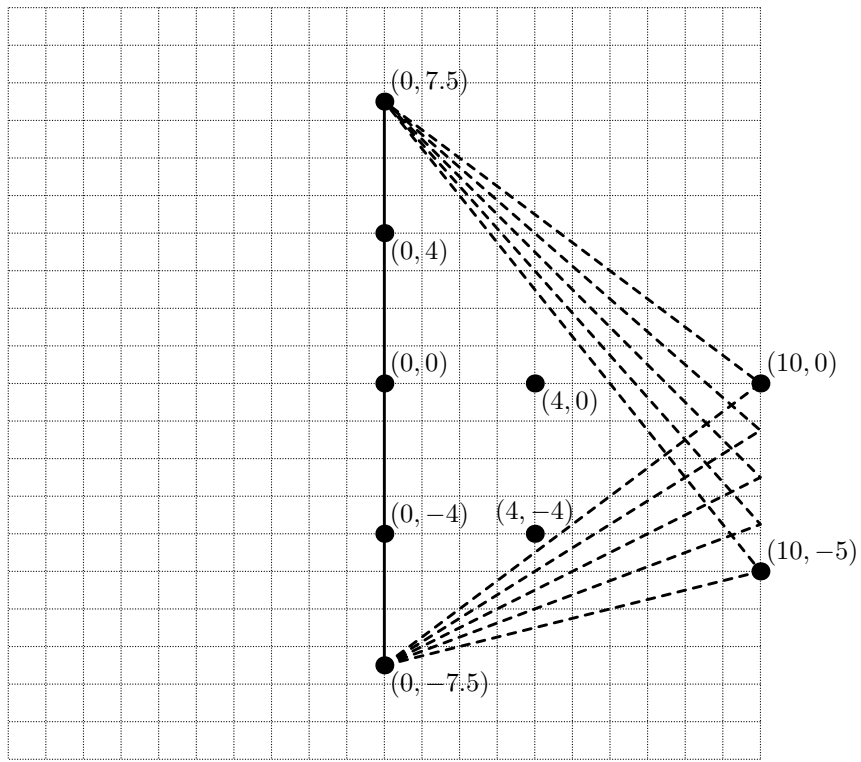}
\end{center}
\caption{The illustration of the containment of the points $(0,4)$, $(0,-4)$, $(4,-4)$ and $(4,0)$
         in the polytope $P$ in the proof of Lemma~\ref{lm:polygon}.}
\label{fig:poly1}
\end{figure}

If $x\le 0$, then the point $(4,4)$ is in the set $P$ as
it is contained in the quadrilateral with the vertices $(0,0)$, $(0,7.5)$, $(-x,10)$ and $(10,-y)$ (see Figure~\ref{fig:poly2} for illustration).
If $x\ge 20$, then the point $(4,4)$ is in the set $P$ as
it is contained in the triangle with the vertices $(0,0)$, $(0,7.5)$ and $(x,-10)$ (note that the points
$(0,7.5)$, $(4,4)$ and $(20,-10)$ lie on the same line as illustrated in Figure~\ref{fig:poly3}).
Finally, if $x\in\left[\frac{20-4y}{8-y},20\right]$,
then the point $(4,-8)$ is contained in the triangle with the vertices $(0,-7.5)$, $(10,-y)$ and $(x,-10)$ and
so it is also contained in $P$ (see Figure~\ref{fig:poly4} for illustration).
Note that we have used that $x\le 20$ as the points $(0,-7.5)$, $(4,-8)$ and $(20,-10)$ lie on the same line (see Figure~\ref{fig:poly3}).
Hence, we can assume that $x\in\left[0,\frac{20-4y}{8-y}\right]$ in the rest of the proof.
In particular, $x\in [0,5/2]$.

\begin{figure}
\begin{center}
\epsfbox{torus10-3.mps}
\end{center}
\caption{The illustration of the containment of the point $(4,4)$ in the polytope $P$
         in the proof of Lemma~\ref{lm:polygon} when $x\le 0$.}
\label{fig:poly2}
\end{figure}

\begin{figure}
\begin{center}
\epsfbox{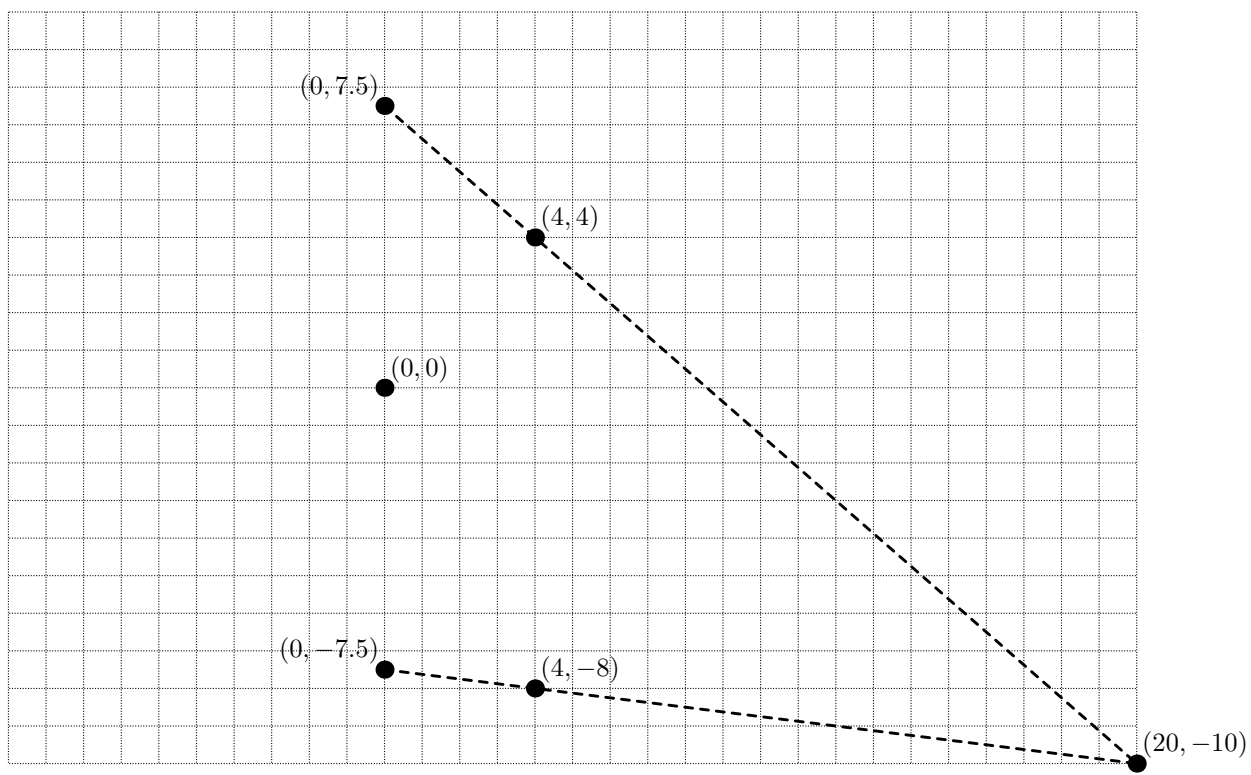}
\end{center}
\caption{The illustration of the mutual position of some of the points in the proof of Lemma~\ref{lm:polygon}.}
\label{fig:poly3}
\end{figure}

\begin{figure}
\begin{center}
\epsfbox{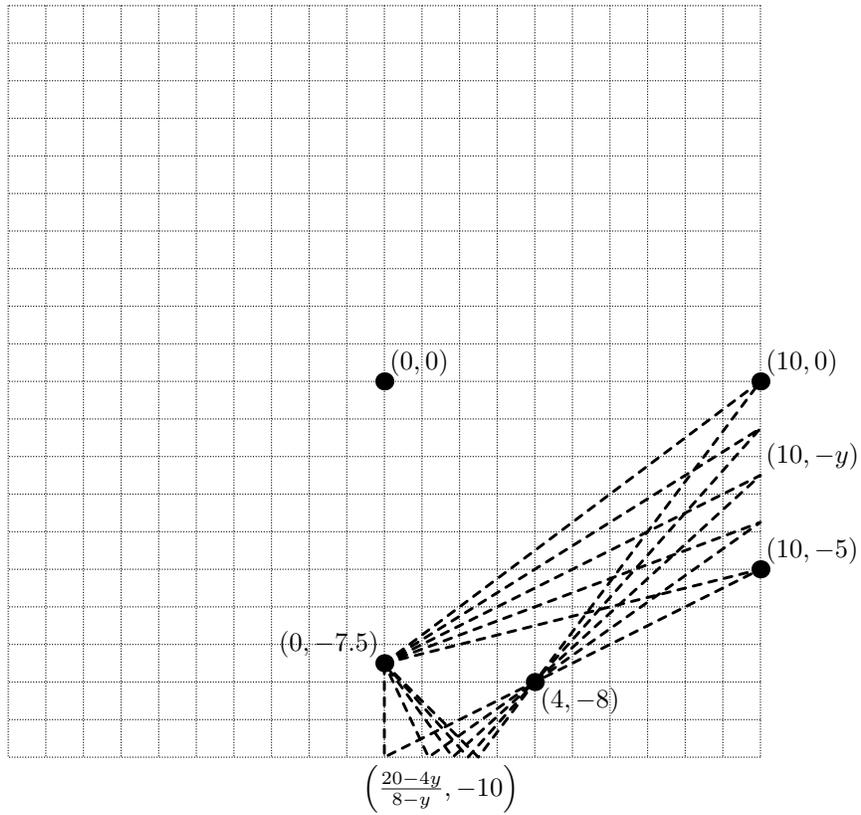}
\end{center}
\caption{The illustration of the containment of the point $(4,-8)$ in the polytope $P$
         in the proof of Lemma~\ref{lm:polygon} when $x\in\left[\frac{20-4y}{8-y},20\right]$.}
\label{fig:poly4}
\end{figure}

We now consider the point $(t,10-t)$.
If $t\ge 40/3$ (see Figure~\ref{fig:poly5} for illustration),
then the point $(4,-8)$ is contained in the triangle with the vertices $(0,0)$, $(x,-10)$ and $(t,10-t)$ (note that
the points $(0,-10)$, $(4,-8)$ and $(40/3,-10/3)$ lie on the same line) and
so the point $(4,-8)$ is contained in $P$.
If $t\in [5,30]$ (see Figure~\ref{fig:poly6} for illustration),
then the point $(4,4)$ is contained in the triangle with the vertices $(-x,10)$, $(0,0)$ and $(t,10-t)$ (note that
the points $(0,0)$, $(4,4)$ and $(5,5)$ lie on the same line and
the points $(-2.5,10)$, $(4,4)$ and $(30,-20)$ lie on the same line) and
so the point $(4,4)$ is contained in $P$.
Finally, if $t\in [0,5]$ (see Figure~\ref{fig:poly7} for illustration),
then the point $(4,4)$ is contained in the triangle with the vertices $(0,0)$, $(t,10-t)$ and $(10,-y)$ (note that
the points $(0,0)$, $(4,4)$ and $(5,5)$ lie on the same line and
the points $(0,10)$, $(4,4)$ and $(10,-5)$ lie on the same line) and
so the point $(4,4)$ is contained in $P$.
Hence, we can assume that $t\le 0$ in the rest of the proof.

\begin{figure}
\begin{center}
\epsfbox{torus10-5.mps}
\end{center}
\caption{The illustration of the containment of the point $(4,-8)$ in the polytope $P$
         in the proof of Lemma~\ref{lm:polygon} when $t\ge 40/3$.}
\label{fig:poly5}
\end{figure}

\begin{figure}
\begin{center}
\epsfbox{torus10-6.mps}
\end{center}
\caption{The illustration of the containment of the point $(4,4)$ in the polytope $P$
         in the proof of Lemma~\ref{lm:polygon} when $t\in [5,30]$.}
\label{fig:poly6}
\end{figure}

\begin{figure}
\begin{center}
\epsfbox{torus10-7.mps}
\end{center}
\caption{The illustration of the containment of the point $(4,4)$ in the polytope $P$
         in the proof of Lemma~\ref{lm:polygon} when $t\in [0,5]$.}
\label{fig:poly7}
\end{figure}

\begin{figure}
\begin{center}
\epsfbox{torus10-8.mps}
\end{center}
\caption{The illustration of the containment of the point $(4,4)$ in the polytope $P$
         in the proof of Lemma~\ref{lm:polygon} when $t\le 0$ and $y\in [0,2]$.}
\label{fig:poly8}
\end{figure}

Since the line passing through the points $(4,4)$ and $(10,-2)$ is parallel to the line $(z,10-z)$, $z\in\RR$,
the triangle with the vertices $(0,0)$, $(t,10-t)$ and $(10,-y)$ contains the point $(4,4)$
whenever $y\le 2$ (note that we use that $t\le 0$ here); see Figure~\ref{fig:poly8} for illustration.
Hence, if $y\in [0,2]$, the polytope $P$ contains the point $(4,4)$, and
we can assume in the rest that $y\in [2,5]$.
If $t\le -1$,
then the point $(4,-8)$ is contained in the triangle with the vertices $(0,-7.5)$, $(-t,t-10)$ and $(10,-y)$ (note that
the points $(1,-11)$, $(4,-8)$ and $(10,-2)$ lie on the same line) and
so the point $(4,-8)$ is contained in $P$; see Figure~\ref{fig:poly9} for illustration.
Hence,
we can assume that $t\in [-1,0]$ in the rest of the proof,
in addition to that $y\in [2,5]$.

\begin{figure}
\begin{center}
\epsfbox{torus10-9.mps}
\end{center}
\caption{The illustration of the containment of the point $(4,-8)$ in the polytope $P$
         in the proof of Lemma~\ref{lm:polygon} when $t\le -1$.}
\label{fig:poly9}
\end{figure}

\begin{figure}
\begin{center}
\epsfbox{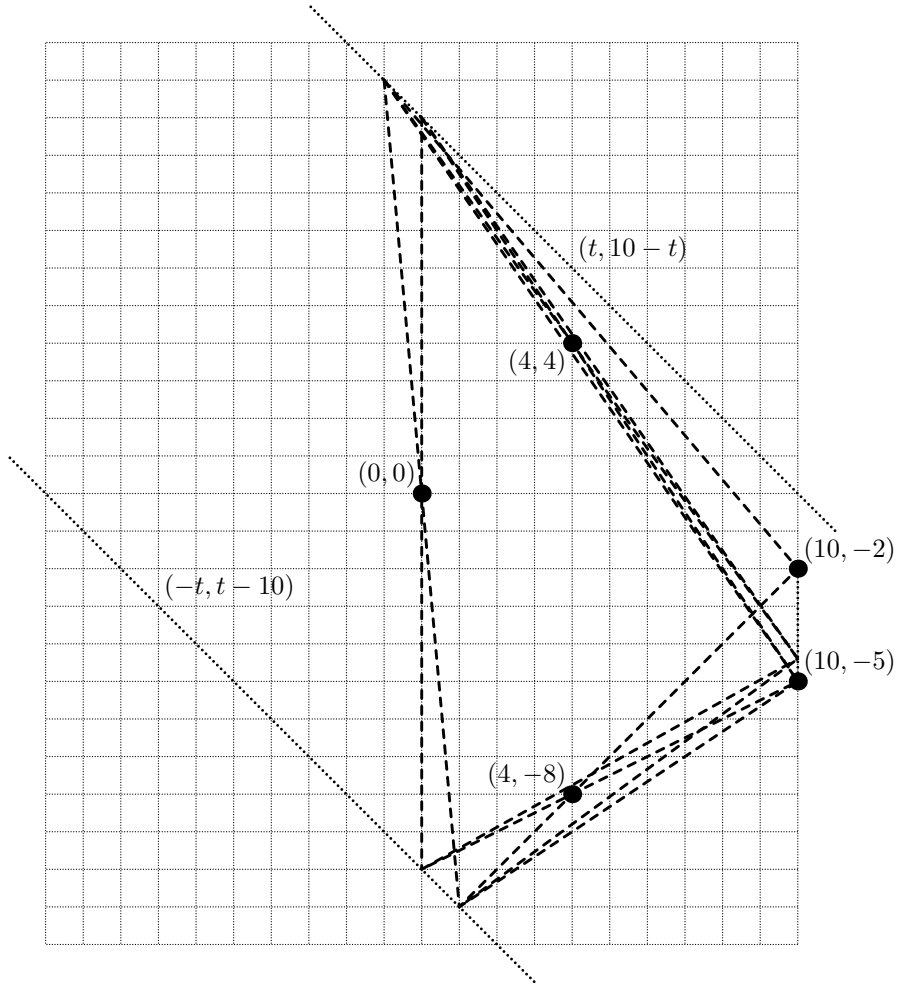}
\end{center}
\caption{The illustration of the containment of one of the points $(4,-8)$ and $(4,4)$ in the polytope $P$
         in the proof of Lemma~\ref{lm:polygon} when $t\in [-1,0]$ and $y\in [2,5]$.}
\label{fig:poly10}
\end{figure}

We next refer to figure~\ref{fig:poly10} for illustration.
Since the point $(10,-5-3t)$ lies on the line containing the points $(-t,t-10)$ and $(4,-8)$,
the triangle with the vertices $(0,0)$, $(-t,t-10)$ and $(10,-y)$ contains the point $(4,-8)$ if $y\ge 5+3t$.
Since the point $\left(10,\frac{2t-20}{4-t}\right)$ lies on the line containing the points $(t,10-t)$ and $(4,4)$,
the triangle with the vertices $(0,0)$, $(t,10-t)$ and $(10,-y)$ contains the point $(4,4)$ if $y\le\frac{20-2t}{4-t}$.
Since it holds that $5+3t\le\frac{20-2t}{4-t}$ for all $t\in [-1,0]$ (the inequality actually holds for all non-positive $t$),
it follows that the polytope $P$ contains the point $(4,-8)$ or the point $(4,4)$ for any $t\in [-1,0]$ and any $y\in [2,5]$.
This concludes the proof of the lemma.
\end{proof}

We are now ready to prove the main theorem of this section.

\begin{theorem}
\label{thm:polygon}
Let $G$ be a graph embedded in the torus with representativity at least $10$.
There exist two free homotopy classes $h_1$ and $h_2$ such that each of the following holds:
\begin{itemize}
\item the graph $G$ contains four disjoint cycles each from $h_1$,
\item the graph $G$ contains four disjoint cycles each from $h_2$,
\item the graph $G$ contains four disjoint cycles each from $h_1+h_2$, and
\item the graph $G$ contains four disjoint cycles each from $h_1-h_2$.
\end{itemize}
\end{theorem}

\begin{proof}
Fix a graph $G$ as in the statement of the theorem.
By Theorem~\ref{thm:polar},
there exist non-zero $(a,b)\in\ZZ^2$ such that $(7.5a,7.5b)\in P(G)$.
By choosing a suitable parameterization of the torus, we may assume that $a=0$.
It follows that the set $P(G)$ contains the point $(0,7.5)$ (note that if $b\ge 2$,
the point $(0.7.5)$ belongs to $P(G)$ as it is a convex combination combination of the point $(0,0)$ and $(0,7.5b)$) and
so the set $P(G)$ also contains the point $(0,-7.5)$.

We now argue that we may assume that $P(G)$ contains the point $(10,-y_0)$ for some $y_0\in [0,5]$.
Since the representativity of $G$ is at least $10$,
it holds by \eqref{eq:maxP} that $c_G(m,n)\ge 10$ for every non-zero $(m,n)\in\ZZ^2$.
In particular, applying \eqref{eq:maxP} with $(m,n)=(1,0)$ yields that
the set $P(G)$ contains a point $(10,z)$ for some $z\in\RR$.
Since reparameterizing the torus using the mapping $(x,y)\mapsto (x,y+\alpha\cdot x\mod 1)$ for $\alpha\in\ZZ$
transforms the set $P(G)$ by the mapping $(x,y)\mapsto (x,y-\alpha\cdot x)$, we can assume that $z\in [-5,+5]$.
Finally, if $z\ge 0$, we reparameterize the torus using the mapping $(x,y)\mapsto (x,-y)$.
We conclude that
we may assume that the set $P(G)$ contains both points $(0.7.5)$ and $(10,z)$ where $z\in [-5,0]$ and
se we can set $y_0=-z$.

Applying \eqref{eq:maxP} with $(m,n)=(0,-1)$ yields that
$P(G)$ contains a point $(x_0,-10)$ for some $x_0\in\RR$, and
applying \eqref{eq:maxP} with $(m,n)=(1,1)$ yields that
$P(G)$ contains a point $(t_0,10-t_0)$ for some $t_0\in\RR$.
Lemma~\ref{lm:polygon} now implies that the set $P(G)$ contains one of the following quadruples of points:
\begin{itemize}
\item the points $(0,4)$, $(4,-4)$, $(4,0)$ and $(4,4)$, or
\item the points $(0,-4)$, $(4,-8)$, $(4,-4)$ and $(4,0)$.
\end{itemize}
In the former case,
we set $h_1$ to be the free homotopy class of $C_{1,0}$ and
$h_2$ to be the free homotopy class of $C_{0,1}$.
Theorem~\ref{thm:Sch} yields that
the graph $G$ has four disjoint cycles each from $h_1$ as $(4,0)\in P(G)$,
four disjoint cycles each from $h_2$ as $(0,4)\in P(G)$,
four disjoint cycles each from $h_1+h_2$ as $(4,4)\in P(G)$, and
four disjoint cycles each from $h_1-h_2$ as $(4,-4)\in P(G)$.
In the latter case,
we set $h_1$ to be the free homotopy class of $C_{1,-1}$ and
$h_2$ to be the free homotopy class of $C_{0,-1}$.
Again, Theorem~\ref{thm:Sch} yields that
the graph $G$ has four disjoint cycles each from $h_1$ as $(4,-4)\in P(G)$,
four disjoint cycles each from $h_2$ as $(0,-4)\in P(G)$,
four disjoint cycles each from $h_1+h_2$ as $(4,-8)\in P(G)$, and
four disjoint cycles each from $h_1-h_2$ as $(4,0)\in P(G)$.
The proof of the theorem is now completed.
\end{proof}

\section{Main result}
\label{sec:main}

We now review concepts related to angle permutations from~\cite[Section 3]{HutRS02}, however,
we prefer casting them in the language of vertex-coloring;
we include short proofs of key statements for completeness.

\begin{figure}
\begin{center}
\epsfbox{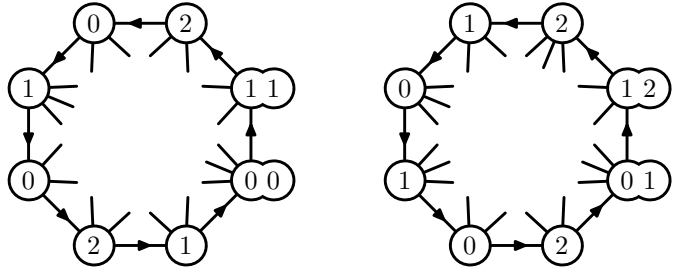}
\end{center}
\caption{An example of the mapping $\varphi_W$ for two closed walks of length eight.
         The vertices $v_0=v_8$ and $v_1=v_9$ are the two right most vertices
	 while the direction of the walk is depicted by arrows.
         For the vertices $v_0=v_8$ and $v_1=v_9$,
	 the left numbers are $\varphi_W(0)$ and $\varphi_W(1)$, and
	 the right numbers are $\varphi_W(8)$ and $\varphi_W(9)$.}
\label{fig:phiW}
\end{figure}

Fix a graph $G$ embedded in an orientable surface,
which will be the plane, the torus or the cylinder in our arguments.
Let $v_1,\ldots,v_{\ell}$ be a closed walk $W$ in $G$.
We say that a vertex $v_i$, $i\in [\ell]$, is \emph{bad}
if $v_i$ is incident with an even number of faces to the left from the walk.
We define $\varphi_W:\{0\}\cup [\ell+1]\to\{0,1,2\}$ as follows:
\[
\varphi_W(k)=\begin{cases}
           k & \mbox{if $k=0$ or $k=1$,} \\
           \varphi_W(k-1) & \mbox{if $k\ge 2$ and $v_k$ is bad, and}\\
           3-\varphi_W(k-1)-\varphi_W(k-2) & \mbox{otherwise,}
	   \end{cases}
\]
where $k\in \{0\}\cup [\ell+1]$ and the indices are modulo $\ell$,
i.e., $v_0=v_{\ell}$ and $v_1=v_{\ell+1}$.
See Figure~\ref{fig:phiW} for an example.
Observe that if $G$ is an Eulerian triangulation of the plane and $v_1,\ldots,v_{\ell}$ a closed walk in $G$,
then the vertex-coloring of $G$ with $0$, $1$ and $2$ such that
the color of the vertex $v_{\ell}=v_0$ is $0$ and the color of $v_1$ is $1$
assigns the color $\varphi_W(i)$ to the vertex $v_i$, $i\in [\ell]$.
The \emph{type} of the closed walk $W$
is the unique permutation $\pi\in S_3$ (we view $S_3$ as permutations on the set $\{0,1,2\}$) such that
$\pi(0)=\varphi_W(\ell)$ and $\pi(1)=\varphi_W(\ell+1)$.
Observe that the type of a closed walk does not depend on the choice of the starting vertex of $W$,
i.e., it depends only on the closed walk itself and its direction.

We next verify that the type of a closed walk in an Eulerian triangulation of the torus depends only on its free homotopy.

\begin{proposition}
\label{prop:type}
Let $G$ be an Eulerian triangulation of the torus.
If two closed walks $W_1$ and $W_2$ are freely homotopic,
then the types of $W_1$ and $W_2$ are the same.
\end{proposition}

\begin{figure}
\begin{center}
\epsfbox{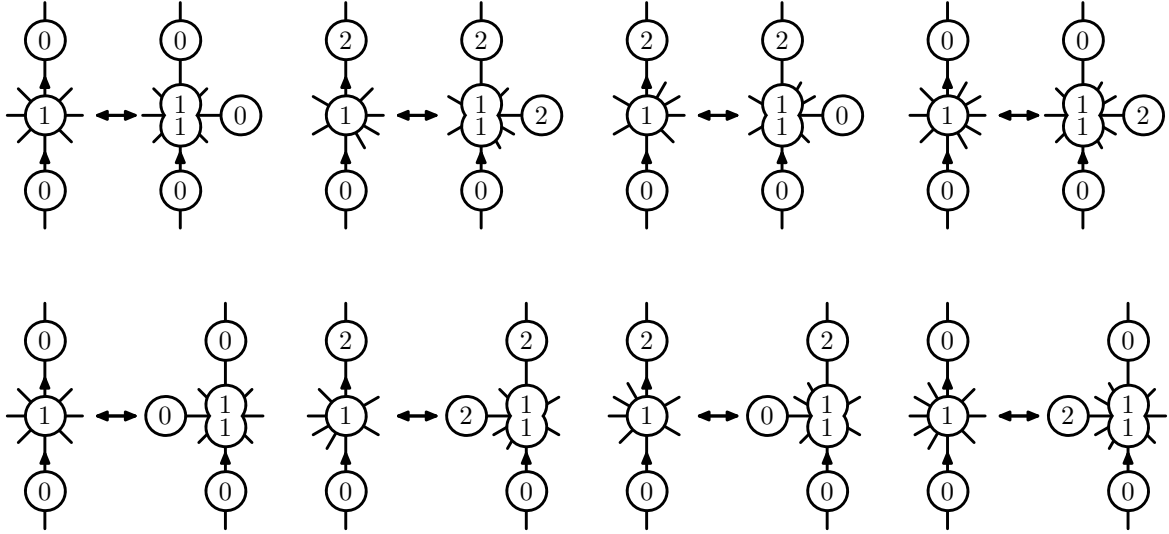}
\end{center}
\caption{Changes of the function $\varphi_W$ for a walk $W$ in an Eulerian triangulation when $a$ is replaced with $aba$.
         The parity of degrees of vertices on a walk $W$ is indicated by the number of half-edges.}
\label{fig:walk1}
\end{figure}

\begin{figure}
\begin{center}
\epsfbox{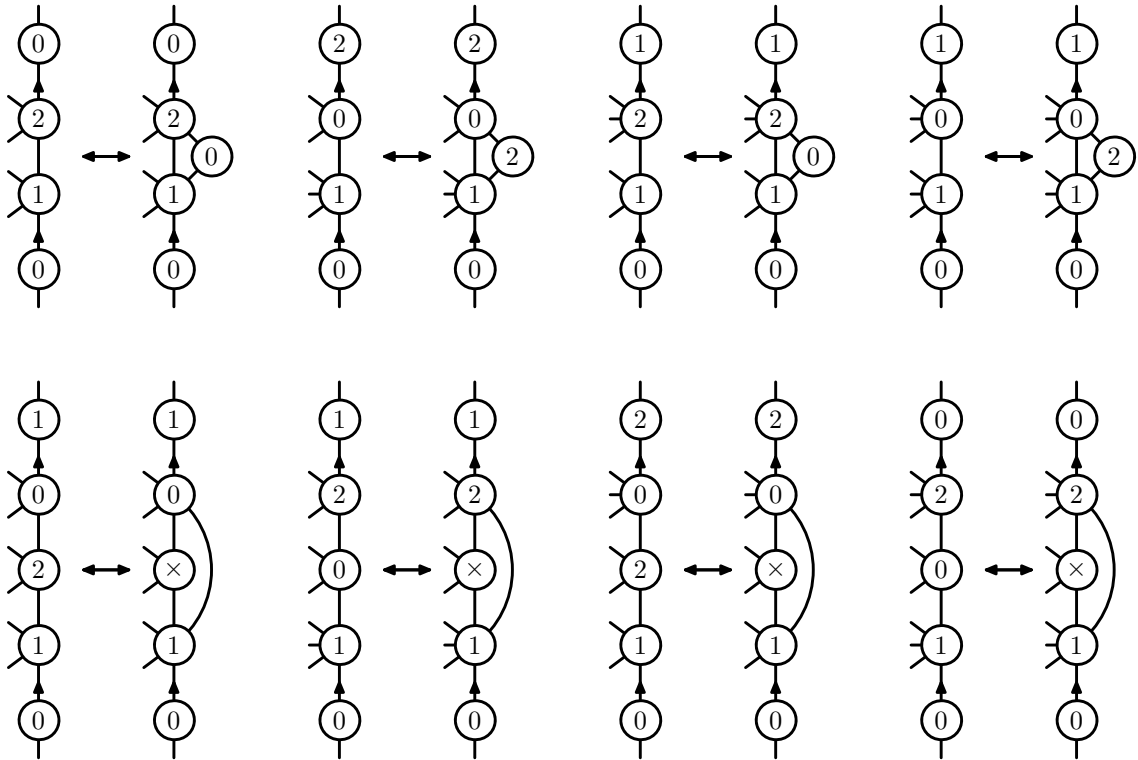}
\end{center}
\caption{Changes of the function $\varphi_W$ for a walk $W$ in an Eulerian triangulation
         when $ab$ is replaced with $acb$ for a triangular face $abc$.
         The parity of degrees of vertices on a walk $W$ is indicated by the number of half-edges;
	 the middle vertex in the bottom four cases is never bad as the triangulation is Eulerian.}
\label{fig:walk2}
\end{figure}

\begin{proof}
Recall that any two freely homotopic closed walks can be transformed to each other
by a sequence of the following two types of operations on subwalks or the reverses of these two operations:
replacing $a$ with $aba$ where $b$ is a neighbor of $a$, and
replacing $ab$ with $acb$ where $abc$ is a triangular face.
Since the former operation preserves the value assigned to the vertex preceding $a$, the vertex $a$ itself, and
the vertex following $a$ in the walk (see Figure~\ref{fig:walk1} for illustration of possible cases) and
the latter operation preserves the value assigned to the vertex preceding $a$, the vertices $a$ and $b$, and
the vertex following $b$ (see Figure~\ref{fig:walk2} for illustration of possible cases),
the type of a closed walk is not altered by any of these two operations.
It follows that the types of any two freely homotopic closed walks are the same.
\end{proof}

Proposition~\ref{prop:type} yields that
we can define a function $\tau$ from free homotopy classes of closed curves to $S_3$
by setting $\tau(h)$ to be the common type of all closed walks of $G$ that belong to $h$.
Consider two closed walks $W_1=v_1,\ldots,v_k$ and $W_2=w_1,\ldots,w_{\ell}$ such that $v_1=w_1$ and $v_k=v_{\ell}$.
Since the type of the closed $v_1,\ldots,v_kw_1,\ldots,w_{\ell}$ is the composition of the types of $W_1$ and $W_2$,
we obtain that $\tau$ is a homomorphism from the group of free homotopy classes to $S_3$.

We next focus on completing a graph obtained from an Eulerian triangulation on the torus to an Eulerian triangulation of the plane.
Recall that a graph obtained by cutting along a non-contractible cycle
is the graph embedded in the surface obtained by cutting the surface along the curve corresponding to the cycle
while duplicating all vertices and edges of the cycle and
embedding the two copies of the cycle along the boundaries created by the cutting.
We remark that the implication stated in Proposition~\ref{prop:cylinder} is actually an equivalence
but we only present the argument for the implication, which we need in the proof of Theorem~\ref{thm:main}.

\begin{proposition}
\label{prop:cylinder}
Let $G$ be an Eulerian triangulation of the torus and
let $C$ be a non-contractible cycle of $G$.
If the type of the cycle $C$ is the identity,
the graph obtained from the Eulerian triangulation $G$ by cutting along $C$,
which is embedded in the cylinder,
can be completed to an Eulerian triangulation of the plane.
\end{proposition}

\begin{proof}
Fix an Eulerian triangulation $G$ of the torus and a non-contractible cycle $C$ of $G$ as in the statement, and
consider the graph obtained from $G$ by cutting along $C$.
Paste two copies of a disk to obtain a plane graph,
let $C_1$ and $C'_1$ be the two copies of the cycle $C$, and
let $f_1$ and $f'_1$ be the two faces bounded by $C_1$ and $C'_1$, respectively.
Observe that all vertices except those on $C_1$ and $C'_1$ have even degrees.

We now show that it is possible to add new vertices and edges inside the face $f_1$ in a way that
all vertices inside the face $f_1$ and all vertices on $C_1$ have even degrees and
all newly created faces inside $f_1$ are triangles.
If the length of $C_1$ is three and its type is the identity,
then none of the three vertices of $C_1$ is bad,
i.e., their degrees are even and no vertices and edges need to be added.
Consider now the mapping $\varphi_{C_1}$ from the definition of a type of a closed walk;
since the type of $C_1$ is identity, the mapping is the same (up to permutation of colors) regardless the choice of a starting vertex.
If the mapping $\varphi_{C_1}$ uses only two colors,
each vertex of $C_1$ is bad, i.e., its degree is odd, and
the length of $C_1$ is even as its type is the identity.
In this case, we add a single new vertex $w$ and join it to each vertex of $C_1$.

\begin{figure}
\begin{center}
\epsfbox{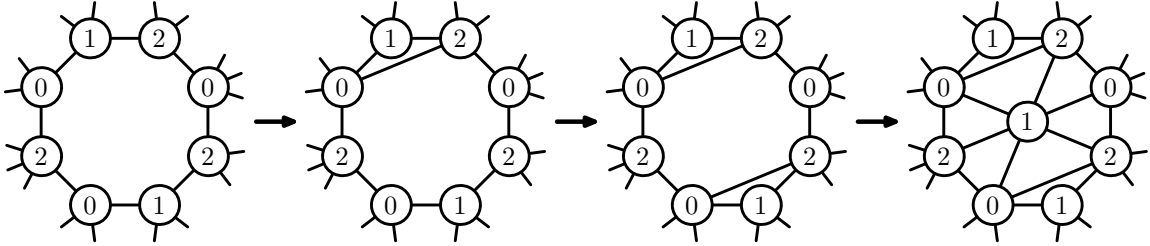}
\end{center}
\caption{An example of the iterative process used in the proof of Proposition~\ref{prop:cylinder}
         to complete the graph inside the face $f_1$.}
\label{fig:cylinder}
\end{figure}

It remains to analyze the case when the length of $C_1$ is at least four and
the mapping $\varphi_{C_1}$ uses three colors.
The latter implies that there exist three vertices $v$, $v'$ and $v''$ consecutive on the cycle $C_1$ such that
the colors $\varphi_{C_1}(v)$, $\varphi_{C_1}(v')$ and $\varphi_{C_1}(v'')$ are distinct;
in particular, the vertex $\varphi_{C_1}(v')$ is not bad and so its degree is even.
We now add the edge $vv''$ and let $C_2$ be the cycle $C_1$ with $vv'v''$ replaced with $vv''$.
Note that $v$ is bad in $C_2$ if and only if $v$ is not bad in $C_1$, and
$v'$ is bad in $C_2$ if and only if $v'$ is not bad in $C_1$.
It follows that $\varphi_{C_2}$ is the restriction of $\varphi_{C_1}$ to the vertices of $C_2$.
Note that the length of $C_2$ is shorter than the length of $C_1$.
We refer to Figure~\ref{fig:cylinder} for an example of the presented procedure.
We now proceed with $C_2$ in the same way as with $C_1$ and
we either obtain a graph embedded in $f_1$ with all degrees even and all faces triangles, or
a shorter cycle $C_3$.
We then proceed with the cycle $C_3$, etc.
Since the length of the cycle shortens at each step,
we eventually obtain a graph embedded within the face $f_1$ such that
all vertices inside $f_1$ and on the cycle $C_1$ have even degrees and
all faces inside $f_1$ are triangles.

A symmetric argument as that used for the face $f_1$ yields that
it is also possible to add new vertices and edges inside the face $f'_1$ in a way that
all vertices inside $f'_1$ and on $C'_1$ have even degrees and all newly created faces are triangles.
It follows that the graph obtained from $G$ by cutting along the cycle $C$
can be completed to an Eulerian triangulation of the plane.
\end{proof}

We are now ready to prove the main theorem of this paper.

\begin{proof}[Proof of Theorem~\ref{thm:main}]
Fix an Eulerian triangulation $G$ of the torus that has representativity at least ten.
We show that there exists a free homotopy class $h_0$ such that
$G$ contains four disjoint cycles each from $h_0$ and $\tau(h_0)$ is the identity.
Apply Theorem~\ref{thm:polygon} to $G$, and
let $h_1$ and $h_2$ be the two free homotopy classes from the statement of the theorem.
If $\tau(h_1)$ or $\tau(h_2)$ is the identity,
then we have established the existence of the sought free homotopy class $h_0$.
Since the group of free homotopy classes is abelian and $\tau$ is its homomorphism to $S_3$,
the subgroup of $S_3$ that is the image of $\tau$ has either two or three elements.
In the former case, $\tau(h_1+h_2)$ is the identity and so we can set $h_0=h_1+h_2$.
In the latter case, $\tau(h_1-h_2)$, $\tau(h_1)$ and $\tau(h_1+h_2)$ are different elements of this $3$-element subgroup and
so one of them is the identity, which yields the choice of $h_0$.

\begin{table}
\begin{center}
\begin{tabular}{|c|cccc|cccc|cccc|}
\hline
$\pi$ & \multicolumn{4}{c|}{$0\mapsto 0$, $1\mapsto 1$, $2\mapsto 2$}
      & \multicolumn{4}{c|}{$0\mapsto 1$, $1\mapsto 0$, $2\mapsto 2$}
      & \multicolumn{4}{c|}{$0\mapsto 1$, $1\mapsto 2$, $2\mapsto 0$} \\
\hline
$k$ & $f_1(k)$ & $f_2(k)$ & $f_3(k)$ & $f_4(k)$ & $f_1(k)$ & $f_2(k)$ & $f_3(k)$ & $f_4(k)$ & $f_1(k)$ & $f_2(k)$ & $f_3(k)$ & $f_4(k)$ \\
\hline
$0$ & $0$ & $0$ & $0$ & $0$  &  $0$ & $1$ & $1$ & $1$  &  $3$ & $3$ & $3$ & $1$ \\
$1$ & $1$ & $1$ & $1$ & $1$  &  $3$ & $3$ & $0$ & $0$  &  $1$ & $1$ & $2$ & $2$ \\
$2$ & $2$ & $2$ & $2$ & $2$  &  $2$ & $2$ & $2$ & $2$  &  $2$ & $0$ & $0$ & $0$ \\
\hline
\end{tabular}
\end{center}
\caption{The choice of functions $f_1,\ldots,f_4$ in the proof of Theorem~\ref{thm:main}.}
\label{tab:fpi}
\end{table}

Fix four vertex-disjoint cycles $C_1$, $C_2$, $C_3$ and $C_4$ of $G$ that belong to the free homotopy class $h_0$.
Let $G'$ be the graph obtained from $G$ by cutting along the cycle $C_1$;
note that the graph $G'$ is embedded in the cylinder.
Let $C'_1$ and $C''_1$ be the two copies of $C_1$ on the boundary of the cylinder.
By Proposition~\ref{prop:cylinder},
the graph $G'$ can be completed to an Eulerian triangulation of the plane and so $G'$ is $3$-colorable.
Fix a vertex-coloring $c$ of $G'$ using the colors $0$, $1$ and $2$.
Observe that the coloring $c$ restricted to $C'_1$ or to $C''_1$
is the function $\varphi_{C_1}$ with the colors $0$, $1$ and $2$ permuted,
however, possibly by a different permutation for each of $C'_1$ or to $C''_1$
Let $\pi\in S_3$ be such that $\pi(c(v''))=c(v')$ for every pair $v'$ and $v''$ of corresponding vertices in $C'_1$ and $C''_1$.
By symmetry,
we may assume that the permutation $\pi$
is either identity, or $\pi(0)=1$, $\pi(1)=0$, $\pi(2)=2$, or $\pi(0)=1$, $\pi(1)=2$, $\pi(2)=0$.
For each of the three cases, define functions $f_1,f_2,f_3,f_4:\{0,1,2\}\to\{0,1,2,3\}$ as given in Table~\ref{tab:fpi}.
For completeness, we also define the function $f_0:\{0,1,2\}\to\{0,1,2,3\}$ as $f_0(i)=i$ for $i\in\{0,1,2\}$.

Let $V_0$ be the set of the vertices of the cycle $C'_1$,
$V_1$ the set of all vertices of $G'$ between the cycles $C'_1$ (excluding this cycle) and $C_2$ (including this cycle),
$V_2$ the set of all vertices of $G'$ between the cycles $C_2$ (excluding this cycle) and $C_3$ (including this cycle),
$V_3$ the set of all vertices of $G'$ between the cycles $C_3$ (excluding this cycle) and $C_4$ (including this cycle), and
$V_4$ the set of all vertices of $G'$ between the cycles $C_4$ (excluding this cycle) and $C''_1$ (including this cycle).
We now recolor each vertex $v$ of every set $V_i$, $i\in [4]$, with the color $f_i(c(v))$.
Observe that since the vertices of $V_0$ keep their original colors,
each vertex $v$ of $V_0$ is colored with $f_0(c(v))$.
We next argue that no two adjacent vertices have the same color.
Consider two adjacent vertices $v$ and $w$ and
consider $i\in [4]$ such that both $v$ and $w$ belong to $V_{i-1}\cup V_i$.
Since the sets $\{f_{i-1}(k),f_i(k)\}$ and $\{f_{i-1}(\ell),f_i(\ell)\}$ are disjoint for any distinct $k,\ell\in\{0,1,2\}$,
the vertices $v$ and $w$ have different colors.
Since any two vertices $v'$ and $v''$ of $C'_1$ and $C''_1$ corresponding to the same vertex of $C_1$
have the same color after recoloring (note that $\pi$ and $f_4$ coincide),
we can glue the cycles $C'_1$ and $C''_1$ to obtain back the graph $G$ and
so we get a proper vertex-coloring of $G$ with four colors.
\end{proof}

\section{Lower bound}
\label{sec:lower}

In this section, we show that the graph $\Gamma(\ZZ_{37},\{1,10,11\})$
is a non-$4$-colorable Eulerian triangulation of the torus with representativity seven.
As mentioned in Section~\ref{sec:intro},
the chromatic number of the graph $\Gamma(\ZZ_{37},\{1,10,11\})$ was determined in~\cite{YehZ03},
however, we prefer including a short proof for completeness.
In addition, we believe that the proof also highlights particular symmetries that the graph $\Gamma(\ZZ_{37},\{1,10,11\})$ possesses.

We formulate and prove the bounds on the representativity and the chromatic number of the graph $\Gamma(\ZZ_{37},\{1,10,11\})$
as propositions.

\begin{proposition}
\label{prop:repr}
The representativity of the unique embedding of the graph $\Gamma(\ZZ_{37},\{1,10,11\})$ in the torus is seven.
\end{proposition}

\begin{proof}
Since the embedding of the graph $\Gamma(\ZZ_{37},\{1,10,11\})$ in the torus is a triangulation,
the representativity is equal to the length of the shortest non-contractible cycle.
Let $v_1,\ldots,v_{\ell}$ be any cycle of $\Gamma(\ZZ_{37},\{1,10,11\})$ and
let $a_i$, $i\in [\ell]$, be one of the six integers $\pm 1$, $\pm 10$ and $\pm 11$ such that
$a_i=v_{i+1}-v_i\mod 37$.
Since $v_1,\ldots,v_{\ell}$ is a cycle in the graph $\Gamma(\ZZ_{37},\{1,10,11\})$,
it holds that
\begin{equation}
a_1+\cdots+a_{\ell}=0\mod 37.
\label{eq:czero}
\end{equation}
Consider a function $h:\{\pm 1,\pm 10,\pm 11\}\to\ZZ^2$ defined as
\begin{align*}
h(+1) & =(+1,0) &
h(+10) & =(0,+1) &
h(+11) & =(+1,+1) \\
h(-1) & =(-1,0) &
h(-10) & =(0,-1) &
h(-11) & =(-1,-1)
\end{align*}
Observe that 
the cycle $v_1,\ldots,v_{\ell}$ is non-contractible if and only if
\begin{equation}
h(a_1)+\cdots+h(a_{\ell})\not=(0,0).
\label{eq:ncontr}
\end{equation}
In particular, the multiset $\{a_1,\ldots,a_{\ell}\}$ determines
whether the cycle $v_1,\ldots,v_{\ell}$ is contractible,
i.e., the order of the elements $a_1,\ldots,a_{\ell}$ is irrelevant.

We will show that any multiset $A$ containing integers $\pm 1$, $\pm 10$ and $\pm 11$ that
satisfies \eqref{eq:czero} and \eqref{eq:ncontr} has at least seven elements (counting multiplicities).
This would imply that the length of the shortest non-contractible cycle of $\Gamma(\ZZ_{37},\{1,10,11\})$ is at least seven and
so the representativity of $\Gamma(\ZZ_{37},\{1,10,11\})$ is at least seven.
Since the cycle formed by the vertices $0$, $1$, $2$, $3$, $4$, $15$ and $26$ is non-contractible,
the representativity of the triangulation $\Gamma(\ZZ_{37},\{1,10,11\})$ is at most seven, and
so establishing the above claim on multisets formed by the integers $\pm 1$, $\pm 10$ and $\pm 11$
would complete the proof of the proposition.

\begin{figure}
\begin{center}
\epsfbox{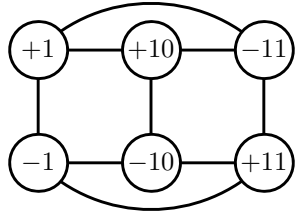}
\end{center}
\caption{The depicted graph has a vertex for each integer $\pm 1$, $\pm 10$ and $\pm 11$, and
two integers are joined by an edge if they cannot be together in a multiset $A$ of minimum size that
satisfies \eqref{eq:czero} and \eqref{eq:ncontr} as argued in the proof of Proposition~\ref{prop:repr}.
In particular, the elements of $A$ must form an independent set in this graph.}
\label{fig:repr-analysis}
\end{figure}

Fix a multiset $A$ formed by the integers $\pm 1$, $\pm 10$ and $\pm 11$ that
satisfies \eqref{eq:czero} and \eqref{eq:ncontr} and that
has the least number of elements (counting multiplicities).
Observe that the multiset $A$ does not contain both $x$ and $-x$ for any $x\in\{1,10,11\}$;
otherwise, a pair of $x$ and $-x$ could be removed from $A$,
which would make the multiset $A$ smaller and preserve \eqref{eq:czero} and \eqref{eq:ncontr}.
In addition,
observe that the multiset $A$ does not contain both $+1$ and $+10$ as
a pair of $+1$ and $+10$ can be replaced with $+11$
making the multiset $A$ smaller while preserving \eqref{eq:czero} and \eqref{eq:ncontr}.
Likewise, the multiset $A$ does not contain both $-1$ and $-10$,
it does not contain both $+1$ and $-11$,
it does not contain both $+10$ and $-11$,
it does not contain both $-1$ and $+11$, and
it does not contain both $-10$ and $+11$.
Hence, we can conclude (see Figure~\ref{fig:repr-analysis}) that
one of the following six cases applies:
the multiset $A$ contains
\begin{itemize}
\item the elements $+1$ and $-10$ only,
\item the elements $+1$ and $+11$ only,
\item the elements $+10$ and $+11$ only,
\item the elements $-1$ and $+10$ only,
\item the elements $-1$ and $-11$ only, and
\item the elements $-10$ and $-11$ only.
\end{itemize}
Since the last three cases are symmetric to the first three cases,
we analyze the first three cases only.

Let $s$ be the sum of the elements of $A$.
Note that $s$ is divisible by $37$ by \eqref{eq:czero}.
In the first case, 
the multiset $A$ contains the elements $+1$ and $-10$ only.
If $s\ge 0$, then $A$ contains the element $+1$ at least ten times, and
if $s\le -74$, then $A$ contains the element $-10$ at least $|s|/10\ge 7$ times.
Finally,
if $s=-37$, then the multiset $A$ contains the element $-10$ at least four times and
the element $+1$ at at least three times,
i.e. $A$ has at least seven elements (which is tight as $+1+1+1-10-10-10-10=-37$).

In the second and third cases, which we analyze jointly,
the multiset $A$ contains the elements $+1$, $+10$ and $+11$ only and so $s>0$.
If $s\ge 74$, then $A$ contains at least $s/11$ elements and
so its size is at least seven (this is tight as $10+10+10+11+11+11+11=74$).
If $s=37$, then $A$ contains at least seven elements equal to $+1$ or $+11$ and
so the size of $A$ is at least seven (this is again tight as $1+1+1+1+11+11+11=37$).
The proof of the proposition is now completed.
\end{proof}

\begin{proposition}
\label{prop:color}
The graph $\Gamma(\ZZ_{37},\{1,10,11\})$ is not $4$-colorable.
\end{proposition}

\begin{proof}
We will prove that the graph $\Gamma(\ZZ_{37},\{1,10,11\})$ has no independent set of size ten,
which implies that the $37$-vertex graph $\Gamma(\ZZ_{37},\{1,10,11\})$ is not $4$-colorable.
Suppose that the graph $\Gamma(\ZZ_{37},\{1,10,11\})$ has an independent set $A$ of size (exactly) ten and
fix such set $A$ for the rest of the proof.

We say that a vertex $x\not\in A$ is \emph{heavy} if $x$ has three neighbors in $A$.
Let $k$ be the number of heavy vertices.
Assign $3$ tokens to each vertex $x\not\in A$ that is heavy and
$2$ tokens to each vertex $x\not\in A$ that is not heavy.
The number of tokens assigned is $54+k$.
Each vertex $x\not\in A$ now sends one token to each neighbor that is in $A$.
In this way, the number of tokens left at each vertex not contained in the set $A$ is non-negative and
each vertex in the set $A$ receives six tokens.
Hence, $k$ must be at least six and so there are at least six heavy vertices.

\begin{figure}
\begin{center}
\epsfbox{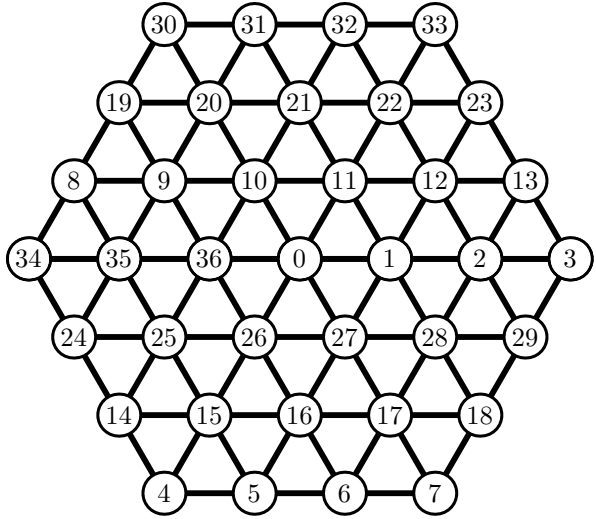}
\end{center}
\caption{The $3$-neighborhood of the vertex $0$ in the graph $\Gamma(\ZZ_{37},\{1,10,11\})$.}
\label{fig:Z37neighborhood}
\end{figure}

Let $z$ be a heavy vertex such that
\begin{itemize}
\item if possible, there is another heavy vertex at distance exactly two from $z$ that has exactly one common neighbor with $z$,
\item if the above is not possible, there is another heavy vertex at distance exactly two from $z$ (and so it has exactly two common neighbors with $z$), and
\item if neither of the above is possible, then choose $z$ to be any heavy vertex.
\end{itemize}
By symmetry of the graph $\Gamma(\ZZ_{37},\{1,10,11\})$,
we may assume that $z=0$ and that the vertices $1$, $10$ and $26$ are contained in the set $A$;
we refer to Figure~\ref{fig:Z37neighborhood} for the structure of the graph $\Gamma(\ZZ_{37},\{1,10,11\})$
on vertices at distance at most three from the vertex $z=0$.
Note that each vertex of $\Gamma(\ZZ_{37},\{1,10,11\})$ is at distance at most three from the vertex $0$.
Since there are at least six heavy vertices,
there must be a heavy vertex at distance at least two from $0$ (otherwise,
only the vertices $0$, $11$, $27$ and $36$ would be heavy).
Let $x$ be a heavy vertex at distance two from $0$ that has the least number of common neighbors with $0$, and
if there is no heavy vertex at distance two from $0$,
let $x$ be any heavy vertex at distance three from $0$.
By symmetry of the graph $\Gamma(\ZZ_{37},\{1,10,11\})$,
we may assume that $x$ is one of the vertices $2$, $12$ and $22$ if its distance from $0$ is two, and
it is one of the vertices $3$, $13$, $23$ and $33$ if its distance from $0$ is three.
However, neither the vertex $13$ nor the vertex $22$ can be heavy as
they have two common neighbors with the vertex $1$, which is in the independent set $A$.
Hence, we can assume in the rest that $x$ is $2$, $12$, $3$, $23$ or $33$.

\begin{figure}
\begin{center}
\epsfbox{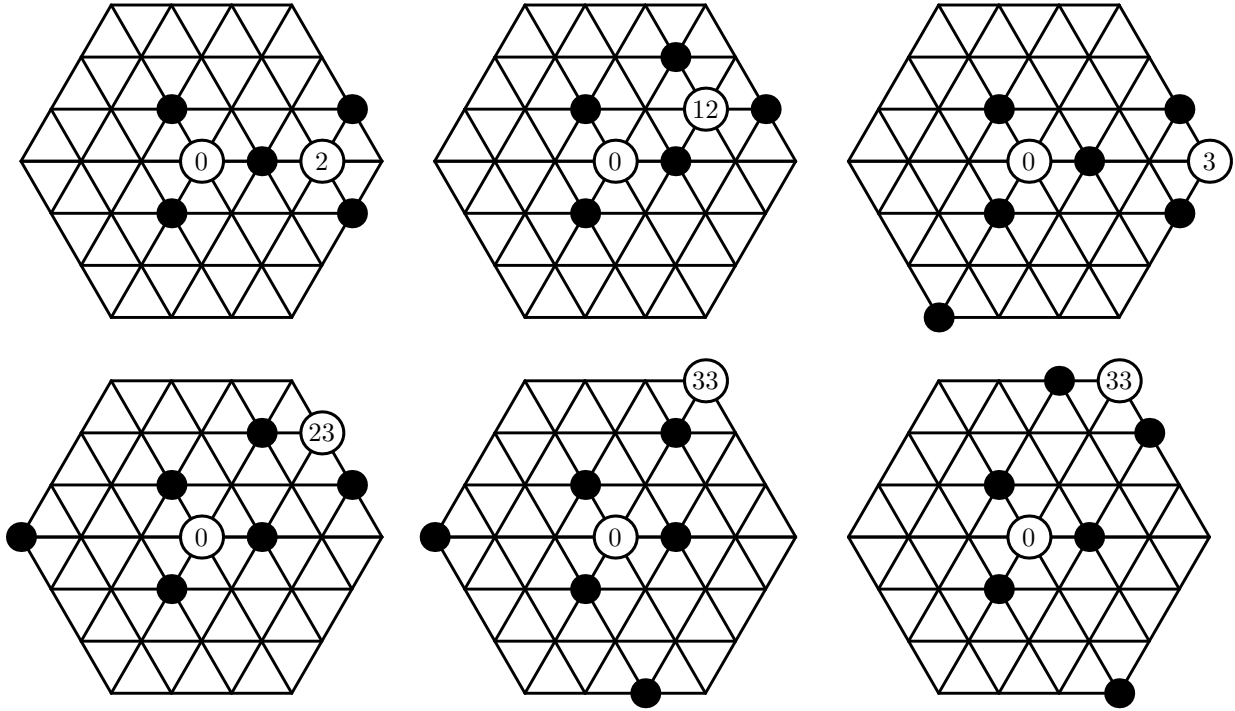}
\end{center}
\caption{Possible configurations of vertices included in an independent set in the graph $\Gamma(\ZZ_{37},\{1,10,11\})$ as
         analyzed in the proof of Proposition~\ref{prop:color}.
         The vertices in the independent set are drawn by full circles and
	 those with three neighbors in the set by empty circles.}
\label{fig:Z37configs}
\end{figure}

Possible configurations around the vertex $x$ are depicted in Figure~\ref{fig:Z37configs}.
Note that if $x$ is $3$ or $23$,
then the vertex $2$ or the vertex $12$, respectively, is heavy,
which is impossible by the choice of $x$ (we would have chosen $x$ to be a heavy vertex at distance two from $0$
if such a vertex existed).
Moreover, if $x=33$ and the fifth configuration in Figure~\ref{fig:Z37configs} applied,
the graph $\Gamma(\ZZ_{37},\{1,10,11\})$ would have two heavy vertices at distance two (specifically,
the vertices $11$ and $33$),
which is impossible by the choice of the vertex $z$ (we should have chosen $z$ to be $11$ or $33$).
Hence, we can assume that the first, second or last configuration depicted in Figure~\ref{fig:Z37configs} applies.

\begin{figure}
\begin{center}
\epsfbox{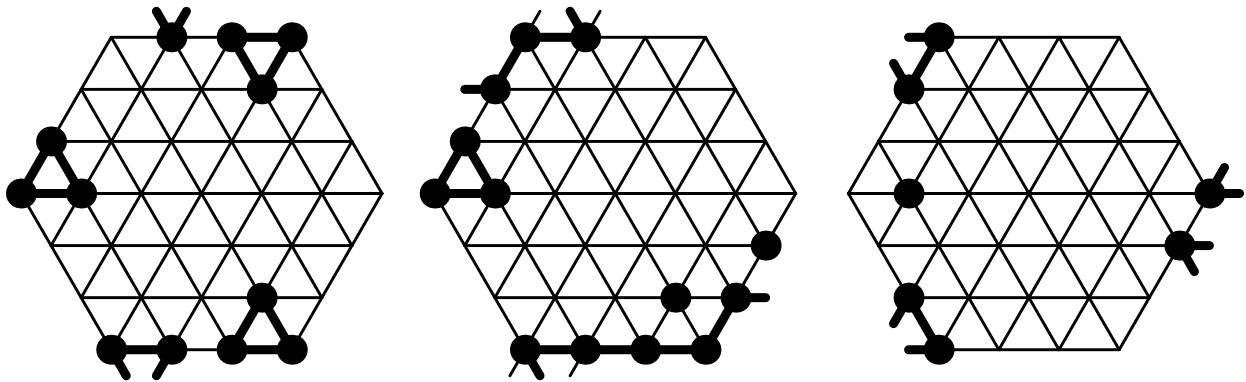}
\end{center}
\caption{The set $B$ from the proof of Proposition~\ref{prop:color}
         when the first, second or last configuration depicted in Figure~\ref{fig:Z37configs} applies (in this order).
	 The vertices of the set $B$ are depicted by bold circles;
	 the edges of the triangles and the cycle of length eight considered
	 in the proof are drawn in bold.}
\label{fig:Z37B}
\end{figure}

Let $A'$ be the set of the vertices of $A$ adjacent to $0$ or $x$, and
let $B$ be the set of all vertices of the graph $\Gamma(\ZZ_{37},\{1,10,11\})$ that
are not contained in $A'$ or adjacent to a vertex of $A'$,
i.e., the set $B$ is the following set in first, second or last configurations, respectively (also see Figure~\ref{fig:Z37B}):
\begin{align*}
B & = \{4, 5, 6, 7, 8, 17, 22, 31, 32, 33, 34, 35\}, \\
B & = \{4, 5, 6, 7, 8, 17, 18, 19, 29, 30, 31, 34, 35\}, \mbox{and} \\
B & = \{3, 4, 14, 19, 29, 30, 35\}.
\end{align*}
Let $H$ be the subgraph of $\Gamma(\ZZ_{37},\{1,10,11\})$ induced by $B$ and
note that $H$ has independence number at least $|A\cap B|$,
i.e., at least $5$ if the first or second configuration applies, and
at least $4$ if the last configuration applies.
However,
if the first configuration applies,
each of the sets $\{4,5,31\}$, $\{6,7,17\}$, $\{8,34,35\}$ and $\{22,32,33\}$
induces a triangle in $H$ and so the independence number of $H$ is at most $4$.
Similarly,
if the last configuration applies,
each of the sets $\{3,4,14\}$ and $\{19,29,30\}$ induces a triangle in $H$ and
since $H$ has seven vertices only,
its independence number is at most $3$.

It remains to analyze the second configuration in Figure~\ref{fig:Z37B}.
By the choice of the vertex $x$, the vertex $2$ is not heavy and so $29\not\in A$.
Similarly, the vertex $27$ is not heavy (otherwise, $12$ and $27$ are heavy vertices
at distance two with exactly one common neighbor), and so $17\not\in A$.
Let $B'=B\setminus\{17,29\}$ and note that $|A\cap B'|\ge 5$.
Observe that the set $\{8,34,35\}$ induces a triangle in $H'$ and
the remaining vertices $4$, $5$, $6$, $7$, $18$, $19$, $30$ and $31$ form a cycle (the vertices are listed
in the order along the cycle).
Since the vertices $4$ and $30$ are adjacent and the vertices $5$ and $31$ are adjacent,
any independent set can contain at most three vertices of the cycle formed by these eight vertices.
It follows that $|A\cap B'|\le 4$ and so the second configuration cannot apply.
\end{proof}

\section{Conclusion}
\label{sec:concl}

\begin{figure}
\begin{center}
\epsfbox{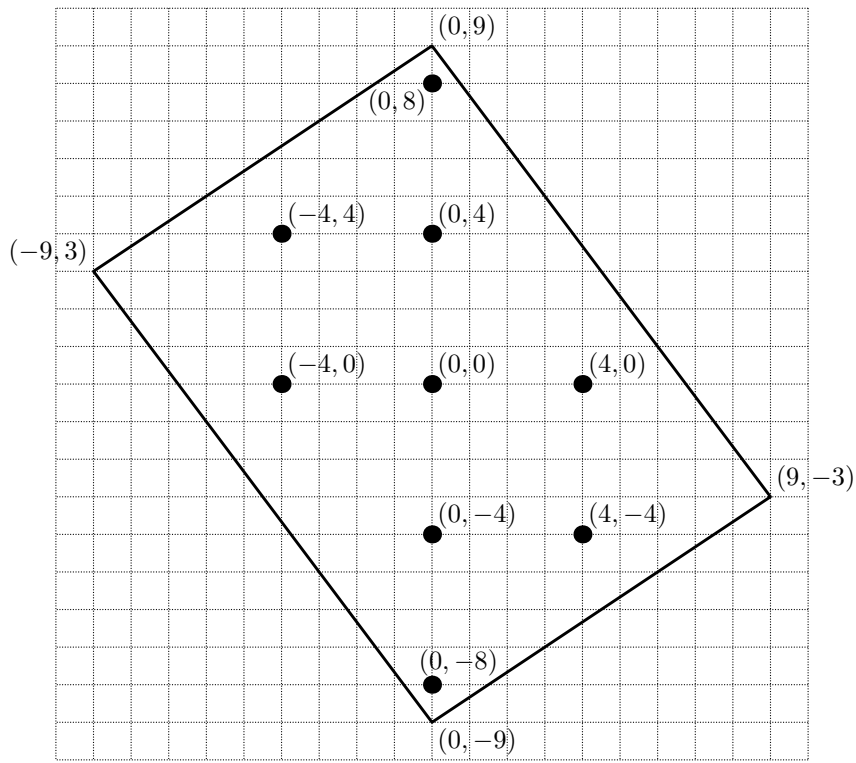}
\end{center}
\caption{A convex set satisfies necessary conditions on being $P(G)$ for a graph $G$ embedded in the torus that
         has representativity $9$ .}
\label{fig:method9}
\end{figure}

Unfortunately,
the techniques used in Sections~\ref{sec:polygon} and~\ref{sec:main}
do not seem to suffice for improving the bound on the representativity in Theorem~\ref{thm:main}
unless a new way of getting insights linking the structure of a graph $G$ and the set $P(G)$ is found.
Indeed, consider the set $X\subseteq\RR^2$ given by the following four linear constraints:
\begin{align*}
4x+3y & \le 27 \\
-2x+3y & \le 27 \\
-4x-3y & \le 27 \\
2x-3y & \le 27
\end{align*}
The set $X$ is visualized in Figure~\ref{fig:method9}.
It is not clear how to rule out that this set
is the set $P(G)$ for an Eulerian triangulation $G$ of the torus with representativity $9$ as
the value of \eqref{eq:maxP} for all non-zero $(m,n)\in\ZZ^2$ is at least $9$.
On the other hand, the set $X$ does not contain any four integer points that
would imply the conclusion of Theorem~\ref{thm:polygon}.
At the same time,
it seems tricky to relate the structure of a possible Eulerian triangulation $G$ to the information encoded by the shape of $P(G)$
in a way that would establish $4$-colorability of $G$
without cutting along four disjoint cycles as in the proof of Theorem~\ref{thm:main}.

Still, we believe that the lower bound given in Section~\ref{sec:lower} is tight, and
so we conjecture that Theorem~\ref{thm:main} can be improved as follows.

\begin{conjecture}
\label{conj}
Every Eulerian triangulation of the torus with representativity at least $8$ is $4$-colorable.
\end{conjecture}

\section*{Acknowledgement}

The authors would like to thank Grzegorz Gu\'spiel and Magdalena Prorok
for many detailed discussions concerning Eulerian triangulations of the torus and their properties.

\bibliographystyle{bibstyle}
\bibliography{torus10}

\end{document}